\documentclass[paper=a4, fontsize=11pt]{scrartcl}	

\usepackage{amsmath,amsfonts,amsthm,amssymb,euscript}			
\usepackage{graphics,graphicx}									

\usepackage{lineno,hyperref}
\modulolinenumbers[1]

\usepackage{color}

\newcommand{\hh}{\mathfrak{H}}

\newcommand{\kk}{\mathfrak{K}}
 
\newcommand{\N}{\mathbb N}

\newcommand{\R}{\mathbb R} 
\newcommand{\SP}{\mathbb S}
\newcommand{\Ss}{\mathcal{S}}
\newcommand{\T}{\mathrm{T}}

\newcommand{\U}{\EuScript{U}}
\newcommand{\V}{\EuScript{V}}
\newcommand{\X}{\mathfrak{X}}
\newcommand{\XX}{\EuScript{X}}
\newcommand{\Y}{\mathfrak{Y}}
\newcommand{\YY}{\EuScript{Y}}

\newcommand{\ZZ}{\EuScript{Z}}

\newcommand{\cotan}{\mbox{cotan}}
\newcommand{\ds}{\displaystyle}
\newcommand{\scs}{\scriptstyle}
\newcommand{\ol}[1]{\overline{#1}}

\newcommand{\dg}[1]{\mbox{deg} \left({#1}\right)}
\newcommand{\ie}{\textit{i.e.~}}
\newcommand{\eg}{\textit{e.g.~}}

\newtheorem{theorem}{Theorem}  
\newtheorem{proposition}{Proposition}
\newtheorem{corollary}{Corollary}

\newtheorem{rem}{Remark}
\newtheorem{definition}{Definition}

\usepackage{amsbsy}

\newcounter{example}[section]

\newcounter{algorithm}[section]

\providecommand{\keywords}[1]{\textbf{\textit{Keywords---}} #1}

\title{Intrinsic and Apparent Singularities in Differentially Flat Systems, and Application to Global Motion Planning}  

\author{Yirmeyahu J. Kaminski\thanks{Department of Applied Mathematics, Holon Institute of Technology, Holon, Israel, \hfill \break e-mail: jykaminski@gmail.com} \and Jean L\'{e}vine\thanks{CAS, Unit\'e Maths et Syst\`emes, MINES-ParisTech, PSL Research University, 60 Bd Saint-Michel, 75272 Paris Cedex 06, France, e-mail: jean.levine@mines-paristech.fr} \and Fran\c{c}ois Ollivier\thanks{LIX, \'Ecole Polytechnique, 91128 Palaiseau Cedex, France, \hfill \break e-mail: Francois.Ollivier@lix.polytechnique.fr}
}

\begin{document}
\date{}
 \maketitle





\begin{abstract}
In this paper, we study the singularities of differentially flat systems, in the perspective of providing global or semi-global motion planning solutions for such systems: flat outputs may fail to be globally defined, thus potentially preventing from planning trajectories leaving their domain of definition, the complement of which we call \emph{singular}. Such singular subsets are classified into two types:  \emph{apparent} and \emph{intrinsic}. A rigorous definition of these singularities is introduced in terms of atlas and local charts in the framework of the differential geometry of jets of infinite order and Lie-B\"acklund isomorphisms. We then give an inclusion result allowing to effectively compute all or part of the intrinsic singularities. Finally, we show how our results apply to the global motion planning of the celebrated example of non holonomic car.
\end{abstract}



\section{Introduction}

Differential flatness has become a central concept in non-linear control theory for the past two decades. See~\cite{FLMR_95,FLMR_99}, the overviews \cite{MMR-ecc,SRA} and \cite{Levine-09} for a thoroughgoing presentation. 

Consider a non-linear system on a smooth $n$-dimensional manifold $X$ given by
\begin{equation}\label{eq::explicit_form}
\dot{x} = f(x,u)
\end{equation}
where $x \in X$ is the $n$-dimensional state vector and $u \in \R^m$ the input or control vector, with $m \leq n$ to avoid trivial situations. 

We consider infinitely prolonged coordinates of the form $(x,\overline{u})\triangleq (x,u,\dot{u}, \ddot{u},\ldots)\in X\times\R^m_{\infty} \triangleq X\times \R^m\times \R^m\times\cdots$ where the latter cartesian product is made of a countably infinite number of copies of $\R^m$.

Roughly speaking, system~\eqref{eq::explicit_form} is said to be (differentially) flat\footnote{This is not a rigorous definition but rather an informal presentation, without advanced mathematics, of the flatness concept. Problems associated to this informal definition are reported in \cite[Section 5.2]{Levine-09}. For a rigorous definition, in the context of implicit systems, the reader may refer to definitions \ref{L-B-equi:def} and \ref{flatness:def} of Section \ref{sec::atlas}.} at a point $(x_0,\overline{u}_0 )\triangleq(x_0,u_0,\dot{u}_0,\ldots) \in X\times \R^m_{\infty}$, if there exists an $m$-dimensional vector $y = (y_1, \dots, y_m)$ satisfying the following statements:
\begin{itemize}
\item $y$ is a smooth function of $x$, $u$ and time derivatives of $u$ up a to a finite order $\beta= (\beta_1, \ldots, \beta_m)$, \ie $y = \Psi(x,u,\dot{u},\dots, u^{(\beta)})$, where $u^{(\beta)}$ stands for $(u_{1}^{(\beta_{1})}, \ldots, u_{m}^{(\beta_{m})})$ and where $u_{i}^{(\beta_{i})}$ is the $\beta_{i}$th order time derivative of $u_{i}$, $i=1, \ldots,m$, in a neighborhood of the point $(x_0,\overline{u}_0)$;
\item $y$ and its successive time derivatives $\dot{y}, \ddot{y}, \dots$ are locally differentially independent in this neighborhood;
\item $x$ and $u$ are smooth functions of $y$ and its time derivatives up to a finite order $\alpha= (\alpha_1, \ldots,\alpha_m)$, \ie $(x,u) = \Phi(y,\dot{y},\dots,y^{(\alpha)})$ in a neighborhood of the point $(y_0,\dot{y}_0,\ldots) \triangleq (\Psi(x_0,u_0,\dot{u}_0,\dots, u_0^{(\beta)}), \dot{\Psi}(x_0,u_0,\dot{u}_0,\dots, u_0^{(\beta+1)}), \ldots) $.
\end{itemize}
Then the vector $y$ is called \emph{flat output}.

Note that it is convenient to regard the above defined  functions $\Phi$ and $\Psi$ as smooth functions over infinite order jet spaces endowed with the product topology\footnote{Recall that in this topology, a continuous function only depends on a finite number of variables, \ie, in this context of jets of infinite order, on a finite number of successive derivatives of $u$ (see \eg \cite[Section 5.3.2]{Levine-09}).} \cite{KLV_86,Zharinov-92,FLMR_99,Levine-09}. They are then called \emph{Lie-B\"acklund isomorphisms} and are inverse one of each other (see \cite{FLMR_99,Levine-09}). However, these functions may be defined on suitable neighborhoods that need not cover the whole space. We thus may want to know where such isomorphisms do not exist at all, a set that may be roughly qualified of \emph{intrinsically singular}, thus motivating the present work: if two points are separated by such an intrinsic singularity, it is intuitively impossible to join them by a smooth curve satisfying the system differential equations and, thus, to globally solve the motion planning problem\footnote{By global motion planning problem, we mean that two arbitrary points of the infinite jet space associated to the system, once the set of intrinsic singularities has been removed, can be joined by a system's trajectory, and thus that this set is connected by arcs.}.

More precisely, the notions of \emph{apparent and intrinsic singularities} are introduced thanks to the construction of an \emph{atlas}, that we call \emph{Lie-B\"acklund atlas},  where \emph{local charts} are made of the open sets where the Lie-B\"acklund isomorphisms, defining the flat outputs, are non degenerated, in the spirit of \cite{CE_14,CE_17} where a comparable idea was applied to a quadcopter model. Intrinsic singularities are then defined as points where flat outputs fail to exist, \ie that are contained in no above defined chart at all. Other types of singularities are called apparent, as they can be ruled out by switching to another flat output well defined in an intersecting chart. Our intrinsic singularity notion may be seen as a generalization of the one  introduced in \cite{Li-Respondek} in the particular case of two-input driftless systems such as cars with trailers, and restricted to the so-called $x$-flat outputs.

Our main result, apart from the above Lie-B\"acklund atlas and singularities definition, then concerns the inclusion of a remarkable and effectively computable set in the set of intrinsic singularities. Note that, since finitely computable necessary and sufficient conditions of non existence of flat output are not available in general \cite{Levine-09,Levine-11}, an easily computable complete characterization of the set of intrinsic singularities is not presently known and it may be useful to label all or part of the singularities as intrinsic thanks to their membership of another set.

To briefly describe this result, we start from the necessary and sufficient conditions for the existence of local flat outputs of meromorphic systems of \cite{Levine-11}\footnote{Other approaches to flatness characterization may be found in \cite{ABMP-ieee-95,Cht,Antritt_2010}}. It consists in firstly transforming the system \eqref{eq::explicit_form} in the locally equivalent implicit form: 
\begin{equation}
\label{eq::implicit_form}
F(x,\dot{x}) = 0
\end{equation} 
where $F$ is assumed meromorphic, and introducing the operator $\tau$, the trivial Cartan field on the manifold of global coordinates $(x,\dot{x}, \ddot{x},\ldots)$, given by $\tau = \sum_{i=1}^n \sum_{j \geq 0} x_i^{(j+1)} \frac{\partial}{\partial x_i^{(j)}}$. Then, we compute the \emph{diagonal} or \emph{Smith-Jacobson decomposition} \cite{Cohn,Levine-09} of the following polynomial matrix: 
\begin{equation}\label{polymat:eq}
P(F) = \frac{\partial F}{\partial x} + \frac{\partial F}{\partial \dot{x}} \tau
\end{equation}
a matrix that describes the variational system associated to \eqref{eq::implicit_form}, and that lies in the ring of matrices whose entries are polynomials in the operator $\tau$ with meromorphic coefficients.

We prove that the set of intrinsic singularities contains the set where $P(F)$ is \emph{not} hyper-regular (see \cite{Levine-09}). As a corollary, we deduce that if an equilibrium point is not first order controllable, then it is an intrinsic singularity. 

These results are applied to the global motion planning problem of the well-known non-holonomic car, which is only used here as a benchmark in order to show how the classical and simple flatness-based motion planning methodology can be extended in presence of singularities. It is also meant to help the reader verifying that the introduced concepts, in the relatively arduous context of Lie-B\"acklund isomorphisms, are nevertheless intuitive and well suited to this situation.

Note that different approaches, also leading to global results, have already been extensively developed in the context of non holonomic systems, based on controllability, Lie brackets of vector fields and piecewise trajectory generation by sinusoids \cite{Murray_93,Jean_96,Chitour_13,Jean_14}, or using Brockett-Coron stabilization results  \cite{Brockett_83,Coron}. However, though some particular nonholonomic systems, as the car example, happen to be flat, our approach applies to the class of flat systems which is different, including \eg pendulum systems, unmanned aerial vehicles and many others that do not belong to the nonholonomic class (see \cite{MMR-ecc,SRA,Levine-09,CE_14,CE_17}).

Remark that, in the car example, the obtained intrinsic singularities are the same as the ones revealed in \cite{Murray_93,Jean_96,Chitour_13,Jean_14} where first order controllability fails to hold, or, according to \cite{Brockett_83,Coron}, where stabilisation by continuous state feedback is impossible. However, the degree of generality of this coincidence is not presently known.

The paper is organized as follows. In section~\ref{sec::atlas}, we introduce the basic language of Lie-B\"acklund atlas and charts. Then this leads to a computational approach for calculating intrinsic singularities. In particular, their links with the hyper-singularity of the polynomial matrix \eqref{polymat:eq} of the variational system is established in Proposition~\ref{hyperreg:prop} and Theorem~\ref{intrinsic:th}, and then specialized in  Corollary~\ref{singularequilibrium:cor} to the case of equilibrium points.

In section~\ref{route:sec}, we apply our results to the non holonomic car. We build an explicit Lie-B\"acklund atlas for this model, compute the set of intrinsic singularities and apply the atlas construction to trajectory planning where the route contains several apparent singularities and starts and ends at intrinsically singular points. Finally, conclusions are drawn in section~\ref{concl:sec}.

\section{Lie-B\"acklund Atlas, Apparent and Intrinsic Singularities}
\label{sec::atlas}

Recall from the introduction that we consider the controlled dynamical system in explicit form \eqref{eq::explicit_form}, where $x$ evolves in some $n$-dimensional manifold $X$. The control input $u$ lies in $\R^m$. Then the system can be seen as the zero set of $\dot{x} - f(x,u)$ in $\T X \times \R^m$, where $\T X$ is the tangent bundle of $X$. From now on, we assume that the Jacobian matrix $\frac{\partial f}{\partial u}(x,u)$ has rank $m$ for every $(x,u)$.

Converting system \eqref{eq::explicit_form} into its implicit form consists in eliminating the input $u$ or, more precisely, in computing its image by the projection $\pi$ from $\T X \times \R^m$ onto $\T X$ to get the implicit relation \eqref{eq::implicit_form},
where we assume that $F: (x,\dot{x}) \in \T X \mapsto \R^{n-m}$ is a meromorphic function, with $m \leq n$. 

Following~\cite{Levine-09,Levine-11}, we embed the state space associated to \eqref{eq::implicit_form} into a diffiety (see~\cite{Zharinov-92}), i.e.  into the manifold $\X \triangleq X \times \R^n_\infty$, where we have denoted by $\R^n_\infty$ the product of a countably infinite number of copies of $\R^n$, with coordinates $\overline{x} \triangleq (x,\dot{x},\ddot{x},\ldots, x^{(k)}, \ldots)$, endowed with the trivial Cartan field: 
$$\tau_{\X} \triangleq \sum_{i=1}^n \sum_{j \geq 0} x_i^{(j+1)} \frac{\partial}{\partial x_i^{(i)}}.$$ 
Note that $\tau_{\X}$ is such that the elementary relations $\tau_{\X} x^{(k)}=x^{(k+1)}$ hold for all $k \in \N$.
The integral curves of both \eqref{eq::explicit_form} and \eqref{eq::implicit_form} thus belong to the zero set of $\{F,\tau_{\X}^{k} F \mid k \in \N\}$ in $\X$. However, there might exist points $\overline{x} =  (x,\dot{x},\ddot{x},\ldots, x^{(k)}, \ldots) \in \X$ such that the fiber $\pi^{-1}(x,\dot{x})$ above $\overline{x}$ is empty, \ie such that there does not exist a $u\in \R^m$ such that $\dot{x}-f(x,u)=0$. We indeed naturally exclude such points. It is easily proven that the integral curves of \eqref{eq::explicit_form} and \eqref{eq::implicit_form} coincide on the set $\X_0$ given by
$$\X_0 = \{\overline{x} \in \X \mid \tau_{\X}^k F(\overline{x}) = 0, \forall k \in \N\} \setminus \{\overline{x} \in \X \mid \pi^{-1}(x,\dot{x}) = \emptyset \}.$$ 
Therefore, the system trajectories are uniquely defined by the triple $(\X, \tau_{\X},F)$ that we call \emph{the system} from now on (see \cite{Levine-09}). Without loss of generality, we may consider that this system is restricted to $\X_0$. 

In order to get rid of any reference to an explicit system, such as the complementary of the empty fibers of the projection $\pi$, we more generally assume that  $\X_0$ is an open dense subset\footnote{As a consequence of the implicit function theorem, the set of points where the fibers are empty is the complement of an open dense subset of the set $\{\overline{x} \in \X \mid \tau_{\X}^k F(\overline{x}) = 0, \forall k \in \N\}$.} of $\{\overline{x} \in \X \mid \tau_{\X}^k F(\overline{x}) = 0, \forall k \in \N\}$.

Let us recall the definitions of Lie-B\"acklund equivalence and local flatness for implicit systems (\cite{Levine-09,Levine-11}):

Consider two systems $(\X, \tau_{\X},F)$ and $(\Y, \tau_{\Y},G)$ where $\Y\triangleq Y\times \R^{q}_{\infty}$, $Y$ being a $q$-dimensional smooth manifold, where $qÊ\in\N$ is arbitrary, with global coordinates $\ol{y} \triangleq (y,\dot{y}, \ldots)$ and trivial Cartan field $\tau_{\Y} \triangleq \sum_{i=1}^q \sum_{j \geq 0} y_i^{(j+1)} \frac{\partial}{\partial y_i^{(j)}}$. As before, we denote
 by $\Y_{0}$ an open dense subset of $\{ \ol{y}\in \Y \mid \tau_{\Y}^{k}G(\ol{y})=0,~\forall k\in \N \}$.

\begin{definition}\label{L-B-equi:def}
We say that $(\X, \tau_{\X},F)$ and $(\Y, \tau_{\Y},G)$ are Lie-B\"acklund equivalent at a pair of points $(\ol{x}_0, \ol{y}_0)\in \X_0\times\Y_0$ if, and only if,
\begin{itemize}
\item[(i)] there exist neighborhoods $\XX_{0}$ of  $\ol{x}_{0}$ in $\X_{0}$, and $\YY_{0}$ of $\ol{y}_{0}$ in $\Y_{0}$, and a one-to-one mapping $\Phi=(\varphi_{0},\varphi_{1},\ldots )$, meromorphic from $\YY_{0}$ to $\XX_{0}$, satisfying 
$\Phi(\ol{y}_{0})=\ol{x}_{0}$ and such that the restrictions of the trivial Cartan fields ${\tau_{\Y}}_{\big \vert \YY_{0}}$ and ${\tau_{\X}}_{\big \vert \XX_{0}}$ are $\Phi$-related, namely 
$\Phi_{\ast}{\tau_{\Y}}_{\big \vert \YY_{0}}={\tau_{\X}}_{\big \vert \XX_{0}}$;
\item[(ii)] there exists a one-to-one mapping $\Psi=(\psi_{0},\psi_{1},\ldots )$, meromorphic from $\XX_{0}$ to $\YY_{0}$, such that $\Psi(\ol{x}_{0})=\ol{y}_{0}$ and $\Psi_{\ast}{\tau_{\X}}_{\big \vert \XX_{0}}={\tau_{\Y}}_{\big \vert \YY_{0}}$.
\end{itemize}
The mappings $\Phi$ and $\Psi$ are called \emph{mutually inverse Lie-B\"acklund isomorphisms} at $(\ol{x}_{0},\ol{y}_{0})$.

The two systems $(\X, \tau_{\X},F)$ and $(\Y, \tau_{\Y},G)$ are called \emph{locally L-B equivalent} if they are L-B equivalent at every pair $(\ol{x}, \Psi(\ol{x}))=(\Phi(\ol{y}),\ol{y})$ of an open dense subset $\ZZ$ of $\X_{0}\times \Y_{0}$, with $\Phi$ and $\Psi$ mutually inverse Lie-B\"acklund isomorphisms on $\ZZ$.
\end{definition}

Accordingly, 
\begin{definition}\label{flatness:def}
The system $(\X, \tau_{\X},F)$ is said (differentially) flat at $\ol{x}_{0}$ if, and only if, it is Lie-B\"acklund equivalent to the trivial system $(\R^{m}_{\infty}, \tau,0)$ at $(\ol{x}_{0},\ol{y}_{0})$ where $\tau$ is the trivial Cartan field on $\R^m_{\infty}$ with global coordinates\footnote{The number of components of $y$ must be equal to $m$ (see \cite{FLMR_99,Levine-09}).} $\ol{y}= (y,\dot{y},\ldots)$, \ie
$\tau = \sum_{i=1}^{m}\sum_{j\geq 0} y_{i}^{(j+1)}\frac{\partial}{\partial y_{i}^{(j)}}$, and where $0$ indicates that there is no differential equation to satisfy. In this case, we say that $y$, or $\Psi$ by extension, is a \emph{local flat output}, well-defined and invertible from a neighborhood of $\ol{x}_{0}$ to a neighborhood of $\ol{y}_0$. 

Finally, the system $(\X, \tau_{\X},F)$ is said locally (differentially) flat if it is flat at every point of an open dense subset $\ZZ$ of $\X_{0}\times \R^{m}_{\infty}$.
\end{definition}

\subsection{Lie-B\"acklund Atlas}\label{LB-atlas:subsec}
From now on, we assume that system~\eqref{eq::explicit_form}, or equivalently \eqref{eq::implicit_form} or, also equivalently, system $(\X, \tau_{\X},F)$ is locally flat.

We now introduce the notion of a Lie-B\"acklund atlas for flat systems. It consists of a collection of charts on $\X_0$, that we call \emph{Lie-B\"acklund charts and atlas}, and that will allow us to define a structure of infinite dimensional manifold on a subset of $\X_0$, that can be $\X_0$ itself is some cases.

\begin{definition}
\begin{itemize}
\item[(i)] A \emph{Lie-B\"acklund chart} on $\X_0$ is the data of a pair $(\U, \psi)$ where $\U$ is an open set of $\X_0$ and $\psi : \U \rightarrow \R^m_{\infty}$ a local flat output,  with local inverse $\varphi: \V \rightarrow \U$ with $\V$ open subset of $\psi(\U)\subset \R^{m}_{\infty}$. 
\item[(ii)] Two charts $(\U_1,\psi_1)$ and $(\U _2,\psi_2)$ are said to be \emph{compatible} if, and only if, the mapping
$$\psi_{1}\circ \varphi_{2}: \psi_{2}(\varphi_{1}(\V_1)\cap \varphi_2(\V_2) ) \subset \R^m_{\infty} \rightarrow \psi_{1}(\varphi_{1}(\V_1)\cap \varphi_2(\V_2) ) \subset \R^m_{\infty}$$ 
is a local Lie-B\"acklund isomorphism (with the same trivial Cartan field $\tau$ associated to both the source and the target) with local inverse $\psi_{2}\circ \varphi_{1}$, as long as
$\varphi_{1}(\V_1)\cap \varphi_2(\V_2) \neq \emptyset$.
\item[(iii)] An \emph{atlas} $\mathfrak{A}$ is a collection of compatible charts. 
\end{itemize}
\end{definition}

For a given atlas $\mathfrak{A} = (\U_i,\psi_i)_{i \in I}$, let $\mathfrak{U}_\mathfrak{A}$ be the union $\mathfrak{U}_\mathfrak{A}\triangleq \bigcup_{i \in I} \U_i$. 

Here our definition differs from the usual concept of atlas in finite dimensional differential geometry, since, on the one hand, diffeomorphisms are replaced by Lie-B\"acklund isomorphisms and, on the other hand, we do not require that $\mathfrak{U}_\mathfrak{A} = \X_0$. The reason for this difference is precisely related to our objective, i.e. identifying the essential singularities of differentially flat  systems. This will become clear in the sequel.  

\subsection{Apparent and Intrinsic Flatness Singularities}
It is clear from what precedes that if we are given two Lie-B\"acklund atlases, their union is again a Lie-B\"acklund atlas. Therefore the union of all charts that form every atlas is well-defined as well as its complement, which we call the set of intrinsic flatness singularities, as stated in the next definition.

\begin{definition}
\label{def::intrinsic-singularities}
We say that a point in $\X_0$ is an \emph{intrinsic flatness singularity} if it is excluded from all charts of every Lie-B\"acklund atlas. Every other singular point, namely every point $\bar{x}\not\in \U_i$ for some chart $(\U_i,\psi_i)$ but for which there exists another chart $(\U_j,\psi_j)$, $j\neq i$, such that $\bar{x}\in \U_j$, is called \emph{apparent}.
\end{definition}

Clearly, this notion does not depend on the choice of atlas and charts. The concrete meaning of this notion is that at points that are intrinsic singularities there is no flat output, \ie the system is not flat at these points. 

On the other hand, points that are apparent singularities are singular for a given set of flat outputs, but well defined points for another set of flat outputs. 

Note, moreover, that obtaining atlases may be very difficult in  general situations and a computable criterion to directly detect intrinsic singularities should be of great help. A simple result in this direction is presented in the following section~\ref{sec::criterion-intrinsic-singularities}.

\subsection{Intrinsic Flatness Singularities and Hyper-regularity}
\label{sec::criterion-intrinsic-singularities}

The purpose of this section is to give a tractable sufficient condition of intrinsic singularity and an algorithm to effectively compute the associated points.

With the notations defined at the beginning of section~\ref{sec::atlas}, we next consider the variational equation, in polynomial form, of system \eqref{eq::implicit_form}:

\begin{equation}\label{polymatgen:eq}
P(F) dx = 0, \quad P(F) = \frac{\partial  F}{\partial x} + \frac{\partial  F}{\partial \dot{x}} \tau_{\X}
\end{equation}
where the entries of the $(n-m)\times n$ matrix $P( F)$ are polynomials in $\tau_{\X}$ with meromorphic functions on $\X$ as coefficients.

Recall that a square $n\times n$ polynomial matrix is said to be \emph{unimodular} if it is invertible and if its inverse is also a matrix whose entries are polynomials in $\tau_{\X}$ with meromorphic functions on $\X$ as coefficients. It is of importance to remark that, according to the fact that the coefficients are meromorphic functions, they are, in general, only locally defined. This local dependence will be omitted unless explicitly needed.

The $(n-m)\times n$ polynomial matrix $P(F)$ is said \emph{hyper-regular} if, and only if, there exists a $(n-m)\times (n-m)$ unimodular  polynomial matrix $V$ and a $n\times n$ unimodular polynomial matrix $U$ such that
\begin{equation}\label{hyper-reg0:eq}
VP(F)U= \left( \begin{array}{cc} I_{n-m}&0_{(n-m)\times m}\end{array}\right).
\end{equation}

In fact, it has been proven in \cite{Antritter-Middeke-09} (see also \cite[Proposition 1]{ACLM-scl}), that the latter definition may be simplified as follows:
\begin{proposition}
The polynomial matrix $P(F)$ is hyper-regular if, and only if, there exists a $n\times n$ unimodular polynomial matrix $U$ such that
\begin{equation}\label{hyper-reg:eq}
P(F)U= \left( \begin{array}{cc} I_{n-m}&0_{(n-m)\times m}\end{array}\right).
\end{equation}
\end{proposition}
 \begin{proof}
$P(F)$ is hyper-regular if, and only if, there are matrices $S$, of size $(n-m)\times (n-m)$ and $T$ of size $n\times n$ such that $SP(F)T =  \left( \begin{array}{cc} I_{n-m}&0_{(n-m)\times m}\end{array}\right)$. Thus, using the identity
\begin{equation*}
\left( \begin{array}{cc} I_{n-m}&0_{(n-m)\times m}\end{array}\right)
  = S^{-1} \left( \begin{array}{cc} I_{n-m}& 0_{(n-m)\times m} \end{array}\right)
\left( \begin{array}{cc} S & 0_{(n-m)\times m} \\ 0_{m\times(n-m)} & I_{m} \end{array}\right)
\end{equation*}
we get
\begin{equation*}
 \left( \begin{array}{cc} I_{n-m}&0_{(n-m)\times m}\end{array}\right) = S^{-1} (SP(F)T) 
 \left( \begin{array}{cc} S & 0 \\ 0 & I_{m} \end{array}\right)
  = P(F)
  \Bigl( T  \left( \begin{array}{cc} S & 0 \\ 0 & I_{m} \end{array}\right)\Bigr) \triangleq P(F)U
\end{equation*}
which proves \eqref{hyper-reg:eq}. The converse is trivial 
\end{proof} 
We say that $P(F)$ is \emph{hyper-singular} at a given point if, and only if, it is not hyper-regular at this point, i.e. if this point does not belong to any neighborhood where $P(F)$ is hyper-regular or, in other words, if at this point no unimodular matrix $U$ satisfying \eqref{hyper-reg:eq} exists.

 Let us denote by $\Ss_{F}$ the subset of $\X_0$ where $P(F)$ is  hyper-singular.  
 The following proposition clarifies some previous results of \cite{Levine-09,Levine-11} in the context of flat systems at a point:
 
 \begin{proposition}\label{hyperreg:prop}
 If system \eqref{eq::implicit_form} is flat at the point $\ol{x}_0 \in \X_0$, then there exists a neighborhood $V$ of $\ol{x}_0$ where $P(F)$ is hyper-regular.
 \end{proposition}
 \begin{proof}
 Assume that system \eqref{eq::implicit_form} is flat at the point $\ol{x}_0\in \X_0$. Then, denoting as before $\ol{y} \triangleq (y,\dot{y}, \ddot{y},\ldots)$ and $\ol{x} \triangleq (x,\dot{x}, \ddot{x},\ldots)$, by definition, there exists a neighborhood $V$ of $\ol{x}_0$ and a flat output $\ol{y} = \Psi(\ol{x}) \triangleq (\Psi_0(\ol{x}), \Psi_1(\ol{x}), \Psi_2(\ol{x}),\ldots) \in \Psi(V) \subset \R^m_{\infty}$ for all $\ol{x} \in V$ and conversely, $\ol{x} = \Phi(\ol{y})\triangleq (\Phi_0(\ol{y}), \Phi_1(\ol{y}), \Phi_2(\ol{y}),\ldots)$ for all $\ol{y} \in \Psi(V)$ such that 
 $F(\Phi_0(\ol{y}), \Phi_1(\ol{y})) = F(\Phi_0(\ol{y}), \tau\Phi_0(\ol{y})) \equiv 0$.

 Taking differentials, we show that $dy$ is a flat output of the variational system. 
 Considering the Jacobian matrix $d\Phi_0(\ol{y})$ (resp. $d\Psi_0(\ol{x})$) of the 0th order component $\Phi_0$ (resp. $\Psi_0$) of $\Phi$ (resp. $\Psi$), we denote by $P(\Phi_0)$ (resp. $P(\Psi_0)$) its polynomial matrix form with respect to $\tau$ (resp. w.r.t. $\tau_{\X}$) (see \cite{Levine-09,Levine-11}).
 
 Since $d\ol{y} = d\Psi(\ol{x})d\ol{x}$ and $d\ol{x}= d\Phi(\ol{y})d\ol{y}$, we get that $dx = P(\Phi_0) dy \in \T^{\ast}V$,  $dy= P(\Psi_0)dx \in \T^{\ast}\Psi(V)$, $P(F)P(\Phi_0)\equiv 0$ and $P(\Phi_0)$ left-invertible, since $P(\Psi_0)P(\Phi_0)= I_m$.
 
 We next consider the Smith-Jacobson decomposition, or diagonal decomposition \cite[Chap. 8]{Cohn}, of $P(F)$: there exists an $(n-m)\times (n-m)$ unimodular matrix $W$, an $n\times n$ unimodular matrix $U$ and an $(n-m)\times(n-m)$ diagonal matrix $\Delta$ such that $WP(F)U= \left( \begin{array}{cc}\Delta&0\end{array}\right)$. Partitionning $U$ into $\left(\begin{array} {cc} U_1&U_2\end{array}\right)$, we indeed get $WP(F)U_1=\Delta$ and $WP(F)U_2=0$, or $P(F)U_2=0$ since $W$ is unimodular. Thus, by elementary matrix algebra, taking account of the independence of the columns of both $U_2$ and $P(\Phi_0)$, one can choose $U$ such that $U_2=P(\Phi_0)$. 
 
Following \cite{Fl-scl2,Levine-11} (see also \cite{ACLM-scl} in a more general context), we introduce the \emph{free} differential module $\kk[dy]$ finitely generated by $dy_1,\ldots, dy_m$ over the ring $\kk$ of meromorphic functions from $\X_0$ to $\R$ and the differential quotient module $\hh \triangleq \kk[dx] / \kk[P(F)dx]$ where $\kk[P(F)dx]$ is the differential module generated by the rows of $P(F)dx$.
Taking an arbitrary non zero element $z= (z_1, \ldots, z_m)$ in $\kk[dy]$, and its image $\xi= P(\Phi_0)z$, we immediately get $P(F)\xi= P(F)P(\Phi_0)z= 0$ which proves that $\xi$ is equivalent to zero  in $\hh$. Since $U = \left(\begin{array} {cc} U_1&P(\Phi_0)\end{array}\right)$ is unimodular, it admits an inverse $V= \left(\begin{smallmatrix} V_1\\V_2\end{smallmatrix}\right)$ and thus $U_1V_1 + P(\Phi_0)V_2= I_n$. Multiplying on the left by $WP(F)$ and on the right by $\xi$, and using the relation $P(F)P(\Phi_0)= 0$, we get $0= WP(F)\xi= WP(F)U_1V_1\xi + WP(F)P(\Phi_0)V_2\xi=WP(F)U_1V_1\xi$. Consequently, recalling that $WP(F)U_1 = \Delta$, we have that $\zeta \triangleq V_1\xi = V_1P(\Phi_0)z$ satisfies $0= WP(F)U_1\zeta= \Delta \zeta$. Consequently, if the entries of the diagonal matrix $\Delta$ contain at least one polynomial of degree larger than 0 with respect to $\tau$, say $\delta_i$ for some $i = 1, \ldots, n-m$, then $\delta_i \zeta_i=0$, and since $\zeta_i\in \kk[dy]$, we have proven that the non zero component $\zeta_i$ is a torsion element of $\kk[dy]$, thus leading to a contradiction with the fact that $\kk[dy]$ is free (see \eg \cite[Theorem 7.3, Chap. III]{Lang_02} or \cite[Corollary 2.2, Chap. 8, Sec. 8.2]{Cohn}). Therefore, the entries of the matrix $\Delta$ must belong to $\kk$, which implies that there exists a submatrix $U'_1$ such that $U' \triangleq \left(\begin{array}{cc}U'_1&P(\Phi_0)\end{array}\right)$ is unimodular and satisfies $WP(F)U'= \left( \begin{array}{cc}I_{n-m}&0\end{array}\right)$, and thus, according to \cite{Antritter-Middeke-09} or \cite[Proposition 1]{ACLM-scl}, that $P(F)$ must be hyper-regular in the considered neighborhood.
\end{proof} 
 
 \begin{rem}
 The above proof may be summarized by the following diagram of exact sequences:
$$
 \begin{array}{ccccccc}
0&\longrightarrow&\R^m_{\infty}& \begin{array}{c}{\scs \Phi}\vspace{-0.5em}\\\longrightarrow\vspace{-0.7em}\\\longleftarrow\vspace{-0.5em}\\{\scs \Psi}\end{array}&\X_0& \stackrel{\scs{F}}{\longrightarrow}&0
\vspace{-0.3em}
\\
&&\hspace{-0.9em}{\scs d} \downarrow& &\hspace{0.5em}\downarrow {\scs d}&&
\vspace{-0.3em}
\\
0&\longrightarrow&\T\R^m_{\infty}&\begin{array}{c} {\scs d\Phi}\vspace{-0.5em}\\\longrightarrow\vspace{-0.7em}\\\longleftarrow\vspace{-0.5em}\\{\scs d\Psi}\end{array}&\T\X_0&\stackrel{\scs{P(F)}}{\longrightarrow} &0
\end{array}
$$

Since $\T\R^m_{\infty}$, is isomorphic to the free differential module  $\kk[dy]$, then $\T\X_0$, that may also be seen as a differential module, is necessarily free. In other words, the kernel of $P(F)$ must be equal to the image of $\T\R^m_{\infty}$ by the one-to-one linear map $d\Phi$, thus sending a basis of $\T\R^m_{\infty}$ (flat outputs) to a basis of $\T\X_0$.
\end{rem}

 \begin{rem}
Due to the Smith-Jacobson decomposition, the hyper-regularity property gives a practical row-reduction algorithm to compute $\Ss_{F}$ (see \cite{Antritter-Middeke-09} and the car example in section~\ref{intrinsic-subsec} below). The hyper-singular set is then deduced by complementarity.
 \end{rem}
 According to Proposition~\ref{hyperreg:prop}, it is clear that on $\Ss_{F}$, the system cannot be flat.
 We thus have the following straightforward result:
 
\begin{theorem}\label{intrinsic:th}
The set $\Ss_{F}$ is contained in the set of flatness intrinsic singularities of the system.
\end{theorem}

In fact (see \cite{Fl-scl2,Levine-09}), $\Ss_{F}$ corresponds to the points where the system is no more  \mbox{F-controllable}, i.e. controllable in the sense of free modules, and therefore non flat (see \cite{CLM_91,FLMR_95,FLMR_99,Levine-09}). 
As a consequence of this theorem, the points where the matrix $P(F)$ is hyper-singular are automatically intrinisic singularities of the system.

Note that, at equilibrium points, F-controllability boils down to first order controllability, \ie controllability of the tangent linear system. 
\begin{corollary}\label{singularequilibrium:cor}
The set made of equilibrium points that are not first order controllable is contained in the set of flatness intrinsic singularities of the system.
\end{corollary}

\section{Applications: Route Planning For the Non Holonomic Car}
\label{route:sec}

In this section, we show on a specific example how the above carried out theoretical analysis applies. 

\subsection{Car Model}
\label{sec::model}

The car (kinematic) model is made of the following set of explicit differential equations (see \eg \cite{Murray_93}):
\begin{equation}\label{carsys:eq}
\left \{ \begin{array}{ccc}
\dot{x} & = & u \cos \theta \\
\dot{y} & = & u \sin \theta \\
\dot{\theta} & = & \frac{u}{l} \tan \varphi
\end{array} \right.
\end{equation}

\begin{figure}[h!]
\begin{center}
\includegraphics[scale=0.35]{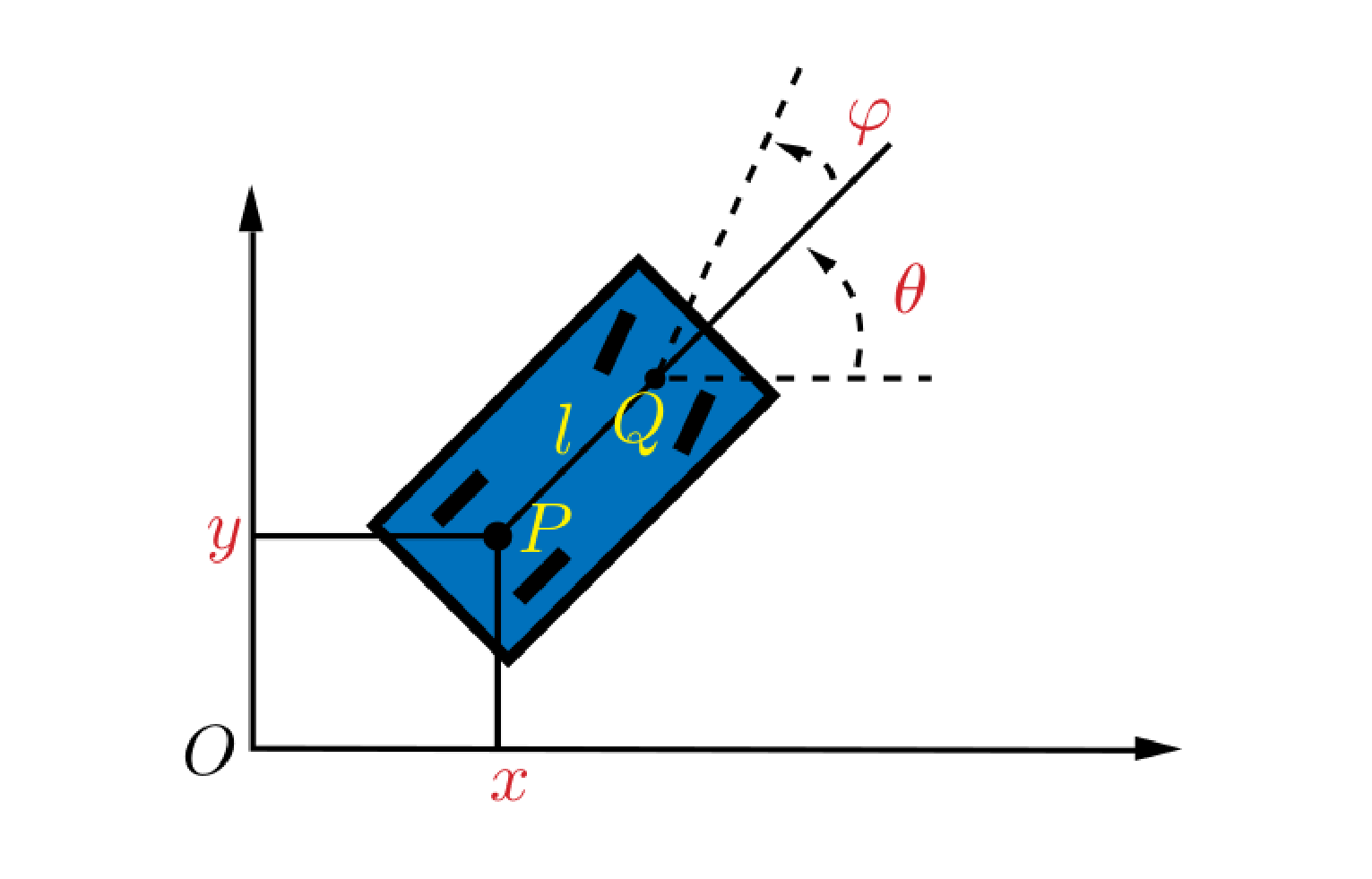}
\caption{Car Model: the state vector is made of the coordinates $(x,y)$ of the rear axle's center and of the angle $\theta$ between the car's axis and the x-axis. The controls are the speed $u$ and the angle $\varphi$ between the wheels' axis and the car's axis. The length $l$ is the distance between the two axles.} \label{fig::car_model}
\end{center}
\end{figure}

Details about the notations are given in the caption of figure~\ref{fig::car_model}. In explicit form, the system evolves in the manifold $\X_1 = \R^2 \times \SP^1 \times \R \times \SP^1$ where the variables are $(x,y,\theta,u,\varphi)$. For the sake of clarity, we note $\X_{11} = \R^2 \times \SP^1$ for the space of state variables $(x,y,\theta)$ and $\X_{12} = \R \times \SP^1$ for the space of control variables $(u,\varphi)$. The tangent bundle of $\X_{11}$ is denoted by $\T \X_{11}$. This system can thus be seen as the zero set in $\T \X_{11} \times \X_{12}$ of the following function:
$$
\mathfrak{F}(x,y,\theta,\dot{x},\dot{y},\dot{\theta},u,\varphi) = \left ( \begin{array}{c}
\dot{x} - u \cos \theta \\
\dot{y} - u \sin \theta \\
\dot{\theta} - \frac{u}{l} \tan \varphi
\end{array} \right )
$$

As in section~\ref{sec::atlas} and again following~\cite{Levine-09,Levine-11}, we consider the local implicit representation of the system, obtained by projecting $\mathfrak F$ on $\T \X_{11}$ by the canonical projection $\pi: \T \X_{11} \times \X_{12} \rightarrow \T \X_{11}$, which amounts to eliminating the controls. In this context, the dynamics \eqref{carsys:eq} are locally equivalent to the zero set of the following function: 
\begin{equation}\label{carimpsys:eq}
F(x,y,\theta,\dot{x},\dot{y},\dot{\theta}) = \dot{x} \sin \theta - \dot{y} \cos \theta = 0.
\end{equation}

We then embed the state space associated to \eqref{carimpsys:eq} into the diffiety $\mathfrak{X} = \R^2 \times \SP^1 \times \R^3_\infty$, endowed with the trivial Cartan field: $\ds \tau_{\X} = \sum_{i=1}^3 \sum_{j \geq 0} x_i^{(j+1)} \frac{\partial}{\partial x_i^{(i)}}$, where we have set $x_1 = x,~ x_2 = y$ and $x_3 = \theta$. 

The system trajectories now live in $\X_0$, the subset of $\{\overline{x}\in \X \mid \tau_{\X}^{k} F =0, \forall k \in \N\}$, where we have excluded the set  $\mathfrak{Z} \triangleq \{(x,y,\theta,\dot{x},\dot{y},\dot{\theta}) \in \T \X_{11}\} \mid \dot{x} = \dot{y} = 0, \dot{\theta} \neq 0\}$ of points of $\T\X_{11}$ where the fibers associated to $\pi$ are empty, \ie the points of $\T \X_{11}$ such that there does not exist $u$ and $\varphi$ such that $F(x, \dot{x})=0$ (see section~\ref{sec::atlas}). Thus
$$\X_0 \triangleq \{\overline{x}\in \X \mid \tau_{\X}^{k} F =0, \forall k \in \N\} \setminus \mathfrak{Z}.$$

\subsection{Lie-B\"acklund Atlas for the Car Model}\label{atlas-car:subsec}

We now define an atlas on $\X_0$ by simply enumerating the charts, as in \cite{CE_14,CE_17} in the context of quadcopters. Each chart is defined on an open set associated to a local Lie-B\"acklund isomorphism $\psi_i$ from $\X_0$ to $\R^2_\infty$ with local inverse denoted by $\phi_i : \R^2_\infty \rightarrow \X_0$. 
For simplicity's sake, we only define $\phi_i$ by its three first components. The other ones are deduced by differentiation, i.e. by applying $\tau_{\X}$ to them an arbitrary number of times. A similar abuse of notation has been used for the definition of $\psi_i$. A point in $\X_0$ is denoted by $\mathfrak{x}$.

\begin{enumerate}
\item Over $U_1 \triangleq \{\dot{x} \neq 0\}$, we take $y_1 = (x,y) = \psi_1(\mathfrak{x})$ and the inverse Lie-B\"acklund transform is given by:
$$
\phi_1 = \left (\begin{array}{c}
x \\ y \\ \tan^{-1}(\frac{\dot{y}}{\dot{x}})
\end{array} \right )
$$

\item Over $U_2 \triangleq \{\dot{y} \neq 0\}$, we take $y_2 = (x,y) = \psi_2(\mathfrak{x})$ and the inverse Lie-B\"acklund transform is given by:
$$
\phi_2 = \left (\begin{array}{c}
x \\ y \\ \cotan^{-1}(\frac{\dot{x}}{\dot{y}})
\end{array} \right )
$$

\item Over $U_3 \triangleq \{\dot{\theta} \neq 0\}$, we take $y_3 = (\theta,x\sin \theta - y\cos \theta) = \psi_3(\mathfrak{x})$. Here for the sake of simplicity, we shall denote $(z_1, z_2)$ the components of $y_3$. In that case the inverse Lie-B\"acklund transform is given by:
$$
\phi_3 = \left (\begin{array}{c}
\frac{\dot{z}_2}{\dot{z}_1} \cos z_1 + z_2 \sin z_1 \\ 
\frac{\dot{z}_2}{\dot{z}_1} \sin z_1 - z_2 \cos z_1  \\
z_1
\end{array} \right )
$$

\item Finally note that the above charts do not contain the set $V = \X_0 \setminus \left( \bigcup_{i=1}^3 U_i \right) = \{\dot{x} = \dot{y} = \dot{\theta} = 0\}$, which corresponds to the set of equilibrium points of the system. Note that, by the definition of $\X_0$, $\dot{x} = \dot{y} = 0$ implies $\dot{\theta} = 0$. Therefore, 
$V= \X_0 \setminus \left( \bigcup_{i=1}^3 U_i \right) = \{\dot{x} = \dot{y} = 0\}$
\end{enumerate}

\

One can check that for all $i,j$, $Im(\phi_i) \subset \X_0$ and that the $\psi_j \circ \phi_i$'s satisfy the compatibility definition of section \ref{LB-atlas:subsec} on $\R^2_\infty$. Therefore we have indeed defined an atlas of $\bigcup_{i=1}^3 U_i  = \X_0 \setminus \{\dot{x} = \dot{y} = 0\}$. Among other things, this allows us to conclude that the car dynamics is globally controllable provided one avoids the singular set $V$, as illustrated in section~\ref{route:sec}.
Note that at this level, we are not able to conclude that the set $\{\dot{x} = \dot{y} = 0\}$ is an intrinsic flatness singularity since, according to definition~\ref{def::intrinsic-singularities} above, we still have to prove that no other atlas can contain this set, hence the importance of the next section based on the results of section~\ref{sec::criterion-intrinsic-singularities}.

\subsection{Flat Outputs and Intrinsic Flatness Singularities of the Car Example}\label{intrinsic-subsec}

One first considers the differential of the implicit equation:
$$
dF = d \dot{x} \sin \theta + \dot{x} \cos \theta d \theta - d \dot{y} \cos \theta + \dot{y} \sin \theta d \theta = (\dot{x} \cos \theta + \dot{y} \sin \theta) d \theta +  \sin \theta d \dot{x} - \cos \theta d \dot{y} 
$$

Note that, if $z$ is an arbitrary variable of the system, we have $d \dot{z} = d (\tau_{\X} z) = \tau_{\X} d z$, \ie the exterior derivative $d$ commutes with the Cartan field $\tau_{\X}$, and the matrix $P(F)$ reads:
$$
P(F) = \left [ \begin{array}{ccc}
(\sin \theta) \tau_{\X} & - (\cos \theta) \tau_{\X} & \dot{x} \cos \theta + \dot{y} \sin \theta
\end{array} \right ] 
$$
thus satisfying 
$$P(F)\left( \begin{array}{c}dx\\dy\\d\theta\end{array}\right)=0$$ 
for all $dx, dy, d\theta$ that are differentials of the variables $x, y, \theta$ satisfying system \eqref{carimpsys:eq}.

Now in the context of the car system given by \eqref{carimpsys:eq}, we are ready to prove the following:

\begin{proposition}\label{intrinsic-car:prop}
The intrinsic singular set of  system \eqref{carimpsys:eq}, given by $\{\dot{x}=\dot{y}=0\}$, is equal to $\Ss_{F}$. 
\end{proposition}
\begin{proof}
We compute the set where $P(F)$ is not hyper-regular. Let us define 
$$A = \dot{x} \cos \theta + \dot{y} \sin \theta.$$ 
Up to a column permutation, $P(F)$ reads  $[A,(\sin \theta) \tau_{\X},- (\cos \theta) \tau_{\X}]$. Then the first column of $U$, say $u_1$ is $u_1 = [1/A,0,0]^t$ (the superscript $^t$ denotes the transposition operator). The second one $u_2$ is  given by $[P_0,P_1,P_2]^t$ where $P_0, P_1, P_2$ are polynomials of $\tau_{\X}$ with \sloppy$\dg{P_{0}} = 1 + \max_{i=1,2}\dg{P_{i}}$, such that $AP_0 + (\sin \theta) \tau_{\X} P_1 - (\cos \theta) \tau_{\X} P_2 = 0$, or $P_0= -\frac{1}{A}\left( (\sin \theta) \tau_{\X} P_1 - (\cos \theta) \tau_{\X} P_2 \right)$. The third column $u_3$ is obtained in the same way: $u_3 = \left[ P'_0,P'_1,P'_2\right]^t$ with $P'_0= -\frac{1}{A}\left( (\sin \theta) \tau_{\X} P'_1 - (\cos \theta) \tau_{\X} P'_2 \right)$ and $P'_1,P'_2$ such that the matrix 
$$\left[ \begin{array}{cc}P_1&P'_1\\P_2&P'_2\end{array}\right]$$ 
is unimodular.
Therefore every decomposition exhibits at least one singularity defined by the vanishing of $A$. Moreover, it is readily seen that the following 0 degree choice $P_1=\sin\theta$, $P_2=-\cos\theta$, $P'_1= \cos\theta$, $P'_2=\sin\theta$ is such that 
$$U=\left[ \begin{array}{ccc} u_1&u_2&u_3\end{array}\right] = 
\left[\begin{array}{ccc} 
\frac{1}{A}&-\frac{1}{A}\tau_{\X}&\frac{\dot{\theta}}{A}\\
0&\sin\theta&\cos\theta\\
0&-\cos\theta&\sin\theta
\end{array}\right]$$
is singular if, and only if, $A=0$. We thus conclude that $P(F)$ is hyper-regular if and only if $A\not=0$.

Finally, the equation $A = \dot{x} \cos \theta + \dot{y} \sin \theta = 0$, combined with $F = \dot{x} \sin \theta - \dot{y} \cos \theta = 0$ leads to $\dot{x} = \dot{y} = 0$. 
We therefore have shown that $\Ss_F = \{\dot{x} = \dot{y} = 0\}$, in other words that the only obstruction to the hyper-regularity of $P(F)$ is a flat output singularity, hence intrinsic according to Theorem~\ref{intrinsic:th}. 
\end{proof}

Note that this direct computation, from the variational system, of the intrinsic singularity confirms that the atlas construction of section~\ref{sec::atlas} was complete in the sense that adding more charts would not reduce the set of intrinsic singularities.

\begin{rem}
Let us stress that the intrinsic singularity obtained in section \ref{atlas-car:subsec} and the planned trajectory of the next section~\ref{route:sec} do not depend on the choice of atlas and charts. Another choice, using \eg  the formulas given in \cite[Section 6.2.4]{Levine-09} would be equally possible, leading to a similar construction.
\end{rem}

\begin{rem} In this example, we could prove that 
$\Ss_{F}$ is in fact equal to the set of intrinsic singularities of the system.
Indeed, it would be most interesting to have an idea of the generality of this situation.
However, examples where  $\Ss_{F}$ does not coincide with the set of flatness intrinsic singularities of the system are not presently known by the authors.
\end{rem}

\subsection{Route Planning}\label{route:sec}

Next, we show how the previously built atlas can be used to control the car over a route along which there are several apparent and intrinsic singularities, as the one depicted in figure~\ref{fig::car_route}. 

\begin{figure}[h!]
\begin{center}
\includegraphics[scale=0.15]{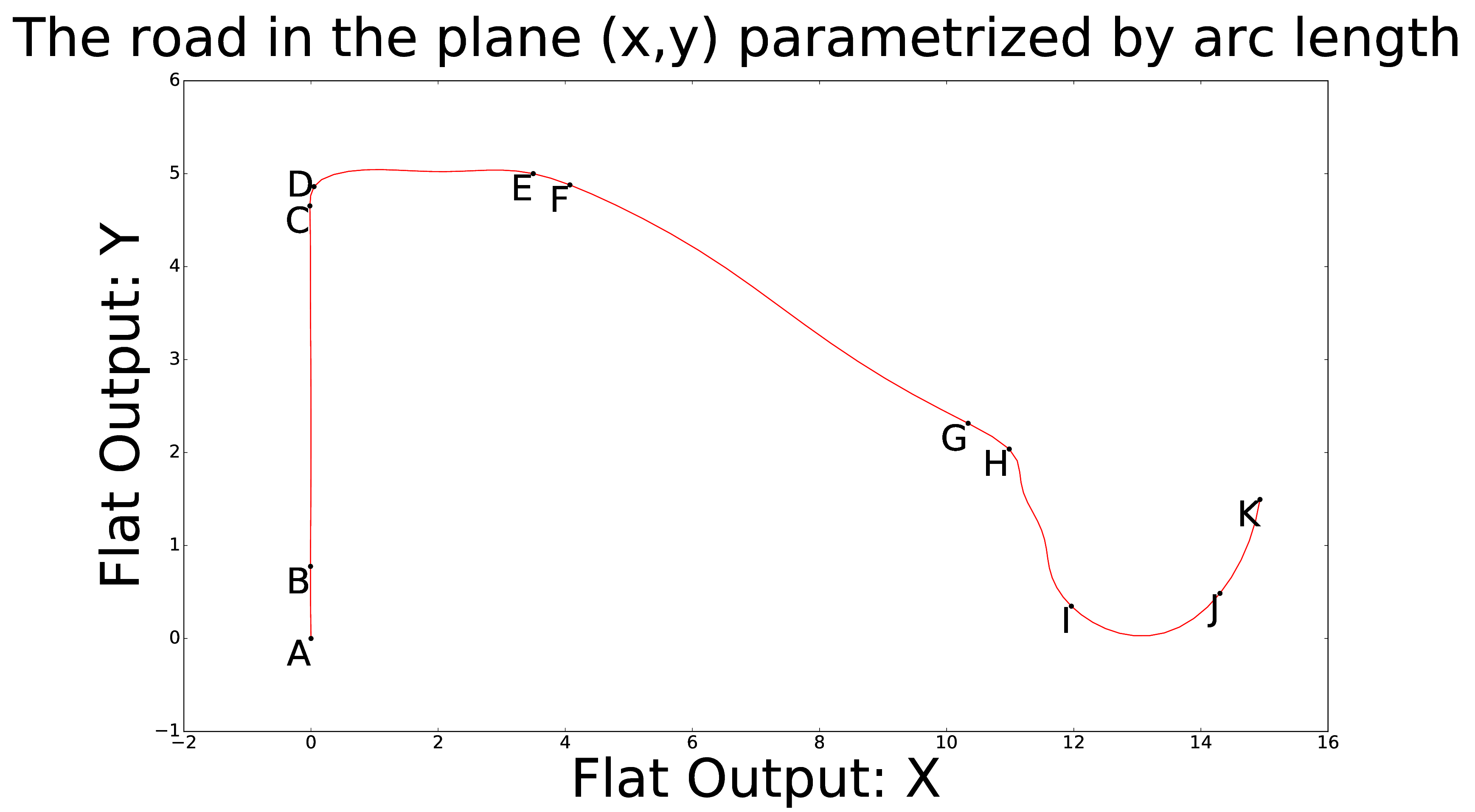}
\caption{Planned car route, parametrized by arc length.} \label{fig::car_route}
\end{center}
\end{figure}

\begin{figure}[h!]
\begin{center}
\includegraphics[scale=0.15]{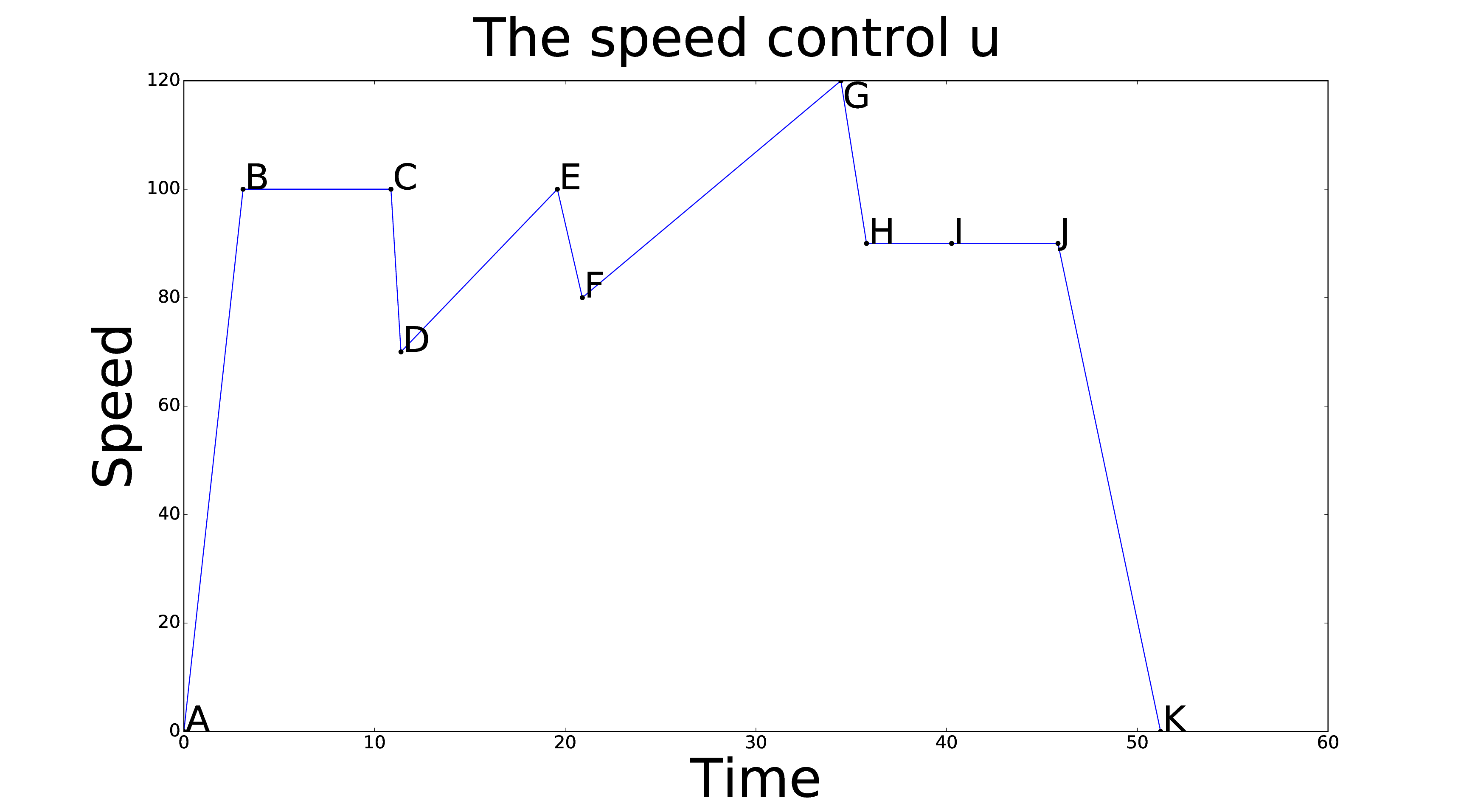}
\caption{The speed 
corresponding to the route depicted in figure~\ref{fig::car_route}} \label{fig::commands}
\end{center}
\end{figure}

\begin{figure}[h!]
\begin{center}
\includegraphics[scale=0.15]{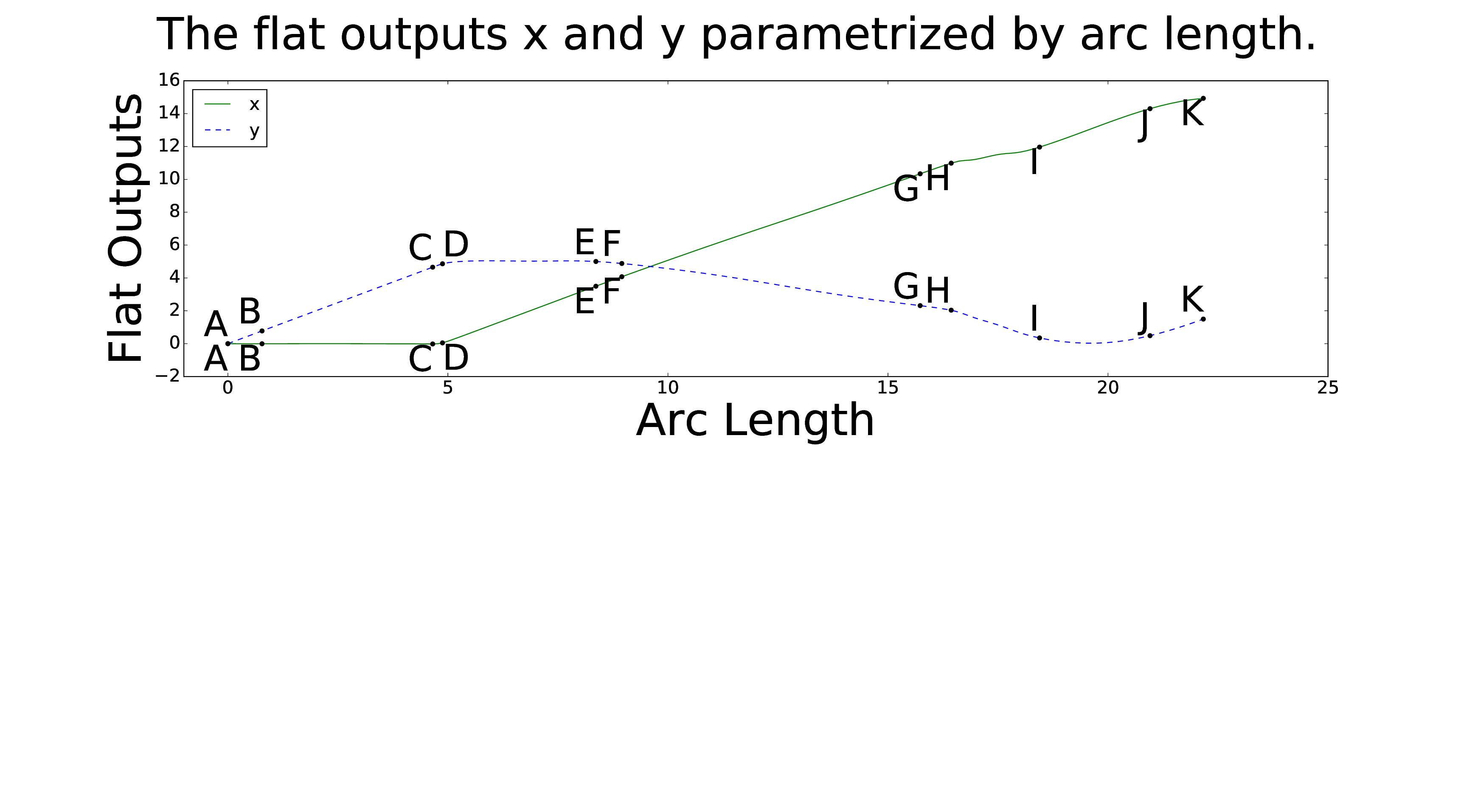}
\includegraphics[scale=0.15]{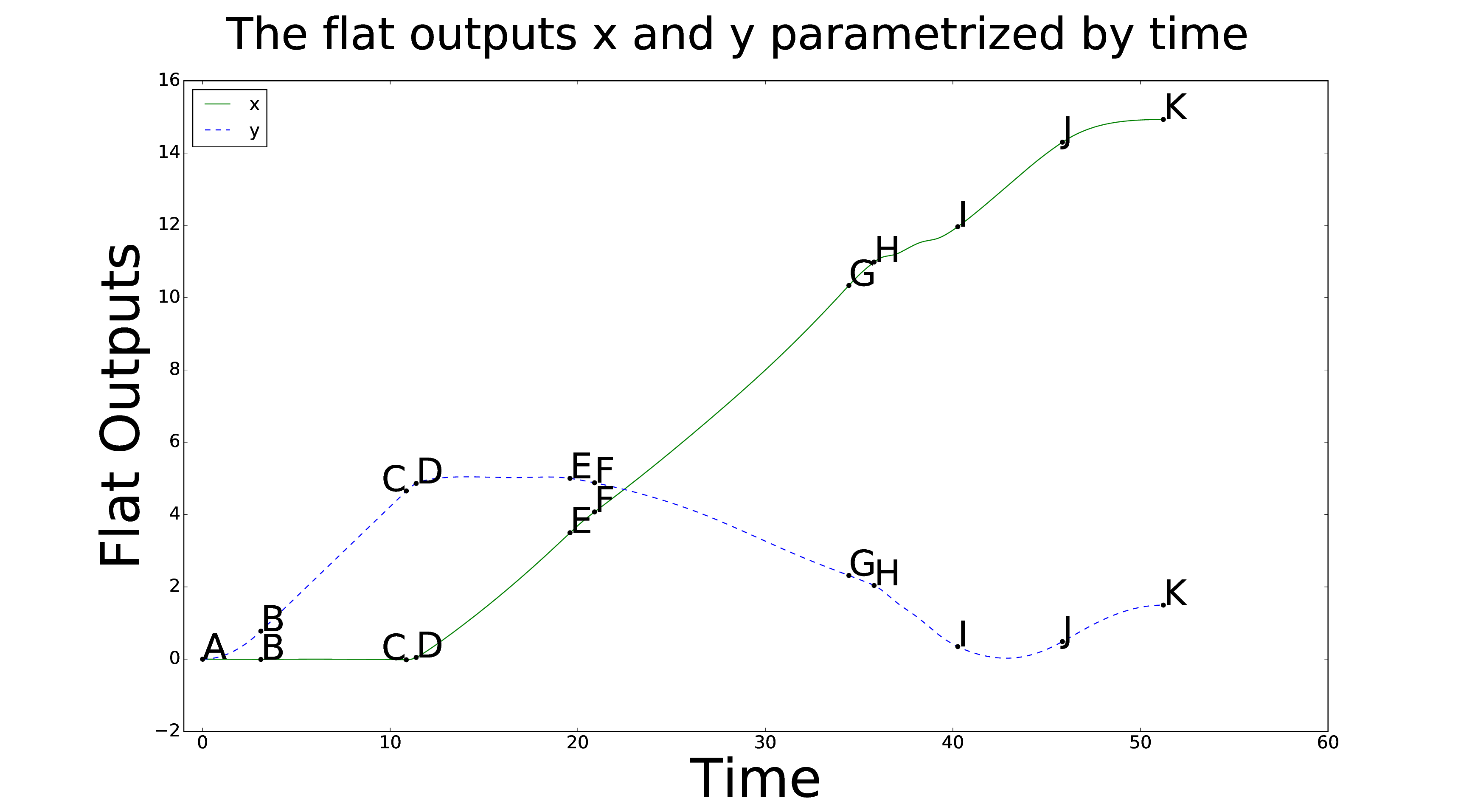}
\caption{The flat outputs parametrized first by arc length and then by time  
corresponding to the route depicted in figure~\ref{fig::car_route}} \label{fig::flat_outputs}
\end{center}
\end{figure}

\begin{figure}[h!]
\begin{center}
\includegraphics[scale=0.15]{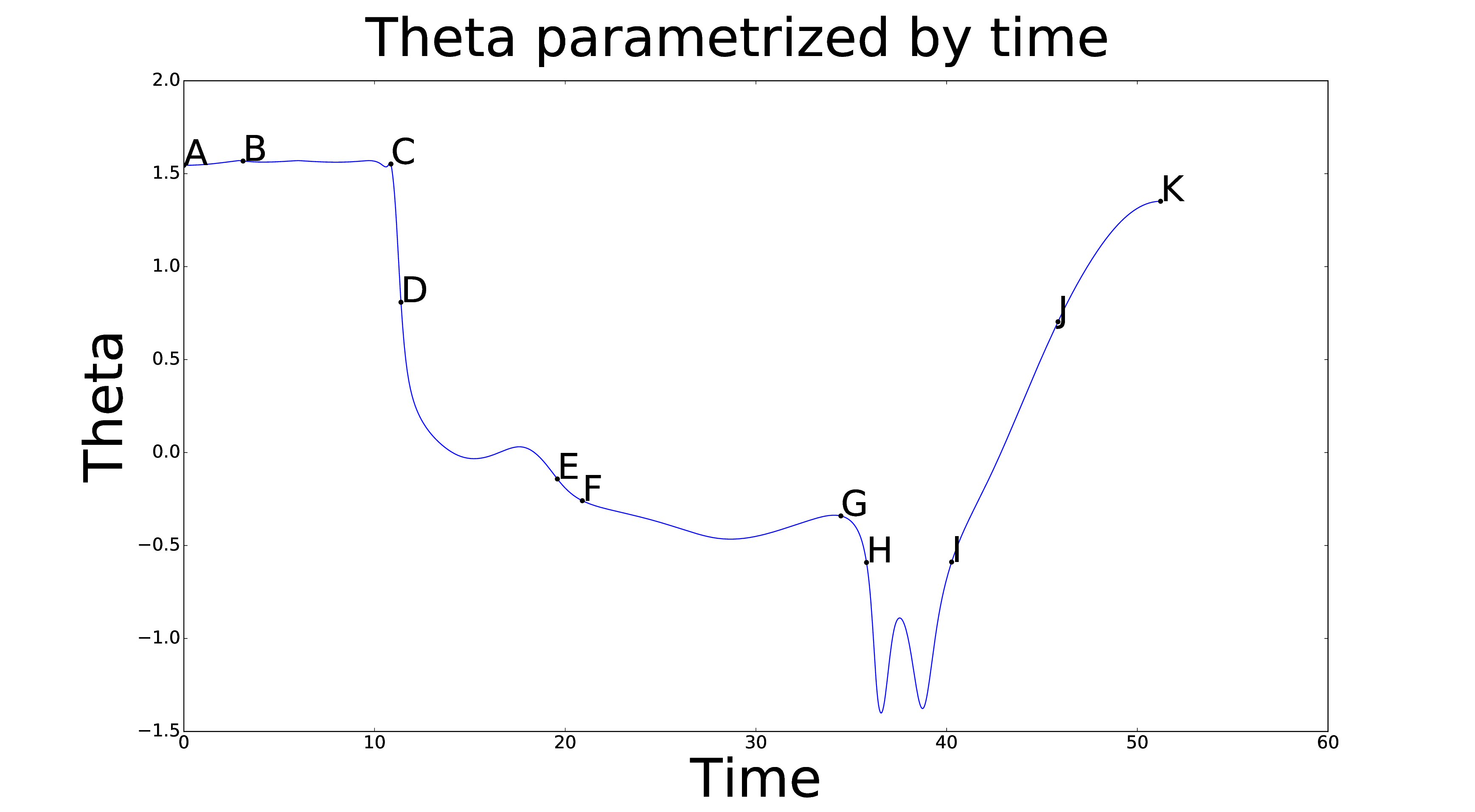}
\includegraphics[scale=0.15]{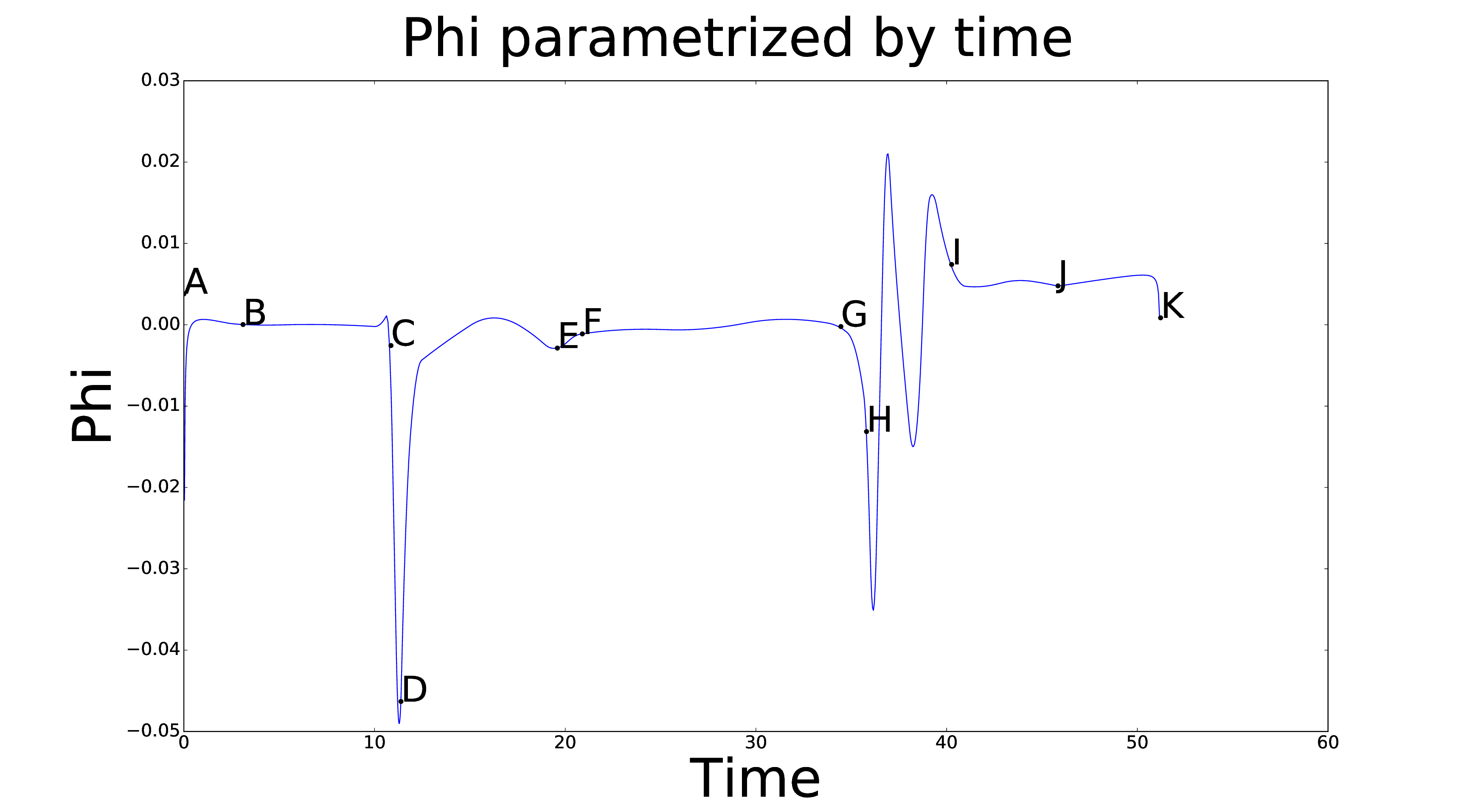}
\caption{The angles theta and phi parametrized by time  
corresponding to the route depicted in figure~\ref{fig::car_route}. For the computation of $\phi$, the car length has been chosen equal to $l=2$m.} \label{fig::theta_phi}
\end{center}
\end{figure}

This route has been defined in several steps. First, the way points $A$, $C$ and following, up to $K$, were chosen in the $(x,y)$-plane to start from the equilibrium point $A$ (intrinsic singularity) along the $y$-axis, which is an apparent singularity for $y_1$ (see section~\ref{atlas-car:subsec}). The car accelerates up to $B$ and then travels at constant speed up to $C$ where it starts making a right turn up to $D$. The route between $C$ and $D$ has been designed by a univariate spline fitting in order to join the previous vertical line to the horizontal segment $DE$, an apparent singularity for $y_2$. The next segment $FG$, after the arc $EF$, again designed by spline fitting, corresponds to a constant heading angle $\theta$, an apparent singularity for $y_3$.
Finally, on the arc $HJ$, the car speed remains constant and then linearly decreases from $J$ to the end point $K$ which is an equilibrium point, thus an intrinsic singularity.

The whole route has been parametrized, in a first step, by its arc length variable on the interval $[0,L]$, with unit speed, in order to allow the design of an arbitrary speed profile over time. 

The trajectory design is done according to the flatness-based method described in \cite{rouchon-et-al-ecc93,rouchon-et-al-cdc93} on each route section. The flat output used is $y_2$ on $AC$, $y_1$ on $CE$, indifferently $y_1$ or $y_2$ on $EG$, and $y_1$ on $GK$ since the component $y$ attains its minimum on this arc, thus with $\dot{y}=0$.

The obtained speed profile of the car is shown in figure~\ref{fig::commands}.    

For the computation of $\phi$, we exclude the end points where the speed vanishes and thus where $\phi$ is only asymptotically defined. See figure~\ref{fig::theta_phi}. Those points, which are indeed intrinsic singularities,  can be approached as close as we want but exactly stopping on them with a prescribed orientation and bounded controls is impossible.

\section{Concluding Remarks}\label{concl:sec}

In this paper, the concepts of intrinsic and apparent flatness singularities have been defined. These notions are of paramount importance for global trajectory planning, namely planning through apparent singularities, avoiding intrinsic singularities, with the possibility of approaching them as close as possible.

We have also shown that intrinsic singularities include a remarkable set, namely the points where the matrix $P(F)$ of the variational system, which plays a major role in the process of flat output computation, is hyper-singular. 

This analysis is illustrated by the global motion planning of a non holonomic car. In this context, we have exhibited an atlas of flat outputs and a complex trajectory safely passing through all possible charts of this atlas.
  
 Note that this approach may be applied in the same way to other flat systems which do not belong to the class of nonholonomic systems. Moreover, it might be possible to extend it to the computation of the largest reachable set of a system.

\section*{References}


\begin{thebibliography}{10}

\bibitem{ACLM-scl}
F. Antritter, F. Cazaurang, J. L\'{e}vine and J. Middeke. On the computation of $\pi$-flat outputs for linear time-varying differential-delay systems, Systems \& Control Letters, 71, 2014, p.  14Ð22.

\bibitem{Antritter-Middeke-09}
F. Antritter and J. Middeke.
An efficient algorithm for checking hyper-regularity of matrices,
ACM Communications in Computer Algebra, 2011, p. 84--86.

\bibitem{Antritt_2010}
F. Antritter and J. L\'evine. Flatness characterization: two approaches, In Advances in the Theory of Control Signals and Systems with Physical Modeling, J. L\'evine and Ph. M\"{u}llhaupt, editors, Lecture Notes in Control and Information Sciences, Springer, Vol. 407, 2011, p. 127--139.

\bibitem{ABMP-ieee-95}
E. Aranda-Bricaire, C.H. Moog and J.-B. Pomet. A linear algebraic framework for dynamic feedback linearization, IEEE Trans. Automat. Control, 40, 1, 1995, p. 127--132.

\bibitem{Brockett_83} R.W.~Brockett. Asymptotic stability and feedback stabilization, in Differential Geometric Control Theory, Birkhauser, 1983, p.~181--191.

\bibitem{CE_14} D.E. Chang and Y Eun. Construction of an atlas for global flatness-based parameterization and dynamic feedback linearization of quadcopter dynamics. Proc. 53rd IEEE Conference on Decision and Control, 2014, p. 686--691.

\bibitem{CE_17} D.E. Chang and Y Eun. Global Chartwise Feedback Linearization of the Quadcopter with a Thrust Positivity Preserving Dynamic Extension. IEEE Trans. Automat. Control, 2017, DOI 10.1109/TAC.2017.2683265.

\bibitem{CLM_91} B. Charlet, J. L\'{e}vine and R. Marino. Sufficient conditions for dynamic state feedback linearization, SIAM J. Control and Optimization, Vol. 29, N.1, 1991, p.38--57.

\bibitem{Cht} V.N. Chetverikov. New flatness conditions for control systems, Proceedings of NOLCOS'01, St. Petersburg, 2001, p. 168--173.

\bibitem{Chitour_13} Y.~Chitour, F.~Jean and R.~Long. A global steering method for nonholonomic systems. J. Differential Equations, 254, 2013, p. 1904--1956.
  
\bibitem{Cohn} P.M. Cohn. Free Rings and Their Relations, Academic Press, London,1985.

\bibitem{Coron} J.M. Coron. Global asymptotic stabilization for controllable systems without drift, Mathematics of Control, Signals and Systems, 5, 1992, p. 295--312.

\bibitem{Fl-scl2} M. Fliess. A remark on {W}illems' trajectory characterization of linear controllability, Systems {\&} Control Letters, 19, 1992, p. 43--45.
	
\bibitem{FLMR_95}
M. Fliess, J. L\'evine, Ph. Martin and P. Rouchon. Flatness and defect of nonlinear
systems: introductory theory and examples, International Journal of Control, Vol. 61, No. 6, 1995, p.1327--1361.

\bibitem{FLMR_99}
M. Fliess, J. L\'evine, Ph. Martin and P. Rouchon. A {L}ie-{B}\"{a}cklund approach to equivalence and flatness of nonlinear systems, IEEE Trans. Automat. Control, Vol. 44, No.5, 1999, p. 922--937.

\bibitem{Jean_96} F.~Jean. The car with n trailers. characterisation of
the singular configurations, ESAIM: Control, Optimisation and Calculus of Variations. Vol.~1, 1996, p.~241--266.

\bibitem{Jean_14} F.~Jean. Control of Nonholonomic Systems: from Sub-Riemannian Geometry to Motion Planning, Springer, 2014.

\bibitem{KLV_86} 
I. S. Krasil'shchik, V. V. Lychagin and A. M. Vinogradov. Geometry of Jet Spaces and Nonlinear Partial Differential Equations, Gordon and Breach, New York, 1986.

\bibitem{Lang_02}
S. Lang. Algebra. Graduate Texts in Mathematics N. 211, Springer, 2002.

\bibitem{Levine-09} 
J. L\'evine. Analysis and Control of Nonlinear Systems: A Flatness-based Approach, Springer, 2009.

\bibitem{Levine-11}
J. L\'evine. On necessary and sufficient conditions for differential flatness, Applicable Algebra in Engineering, Computation and Communication, 2011, p. 27--90.

\bibitem{Li-Respondek}
S.J. Li and W. Respondek. Flat outputs of two-inputs driftless control systems, ESAIM: Control, Optimisation and Calculus of Variations, 18, 2012, p. 774--798.

\bibitem{MMR-ecc}
Ph. Martin, R.M. Murray and P. Rouchon. Flat systems, Plenary Lectures and Minicourses, Proc. ECC 97, Brussels, G. Bastin and M. Gevers Eds., 1997, p. 211--264.

\bibitem{Murray_93} 
R.M.~Murray and S. Sh.~Sastry. Nonholonomic motion planning: steering using sinusoids,
IEEE Transactions on Automatic Control, Vol.~38, No.~5, 1993, p. 700--716.

\bibitem{rouchon-et-al-ecc93}
P. Rouchon, M. Fliess, J. L\'{e}vine and Ph. Martin. Flatness and motion planning{:} the car with n-trailers, Proc. ECC'93, Groningen, 1993, p.~1518--1522.

\bibitem{rouchon-et-al-cdc93}
P. Rouchon, M. Fliess, J. L\'{e}vine and Ph. Martin. Flatness, motion planning and trailer systems. Proc. IEEE Conf. Decision and Control, San Antonio, 1993, p. 2700--2705.

\bibitem{SRA} 
H. Sira-Ramirez and S.K. Agrawal. Differentially Flat Systems, Marcel Dekker, New York, 2004.

\bibitem{Zharinov-92}
V.V. Zharinov. Geometric Aspects of Partial Differential Equations, World Scientific, 1992.

\end{thebibliography}

\end{document}